\documentclass[12pt,reqno]{amsart}
\usepackage{amsmath,amsthm,amssymb,amsfonts,amscd}
\usepackage{mathrsfs}
\usepackage{bbding}
\usepackage{graphicx,latexsym}
\usepackage{hyperref}

\numberwithin{equation}{section} 

\theoremstyle{plain}
\newtheorem{theorem}{Theorem}[section]
\newtheorem{lemma}{Lemma}[section]

\newtheorem{proposition}{Proposition}[section]

\theoremstyle{definition}

\newtheorem{remark}{Remark}

\renewcommand{\Re}{\operatorname{Re}}
\renewcommand{\Im}{\operatorname{Im}}

\begin{document}

\title{Strong orthogonality between the M\"obius function and nonlinear exponential functions in short intervals}
\author{Bingrong Huang}
\address{School of Mathematics \\ Shandong University \\ Jinan \\Shandong 250100 \\China}
\email{brhuang@mail.sdu.edu.cn}
\date{\today}

\begin{abstract}
  Let $\mu(n)$ be the M\"{o}bius function, $e(z) = \exp(2\pi iz)$, $x$ real and $2\leq y \leq x$.
  This paper proves two sequences $(\mu(n))$ and $(e(n^k \alpha))$ are strongly orthogonal in short intervals.
  That is, if $k \geq 3$ being fixed and $y\geq x^{1-1/4+\varepsilon}$, then for any $A>0$, we have
  \[
    \sum_{x< n \leq x+y} \mu(n) e\left( n^k \alpha \right) \ll y(\log y)^{-A}
  \]
  uniformly for $\alpha \in \mathbb{R}$.
\end{abstract}

\maketitle
\tableofcontents

\section{Introduction}

Let $\mu(n)$ be the M\"{o}bius function, $e(z) = \exp(2\pi iz)$, $k\geq 1$ an integer,
$x$ real and $2 \leq y \leq x$.
The following classical result proved by Davenport \cite{davenport1937some} for $k=1$ and
by Hua \cite{hua1965additive} for $k\geq 2$:
for any $A>0$, we have
\begin{equation} \label{eqn: davenport}
  \sum_{n\leq x} \mu(n) e(n^k \alpha)  \ll  x (\log x)^{-A}
\end{equation}
uniformly for $\alpha \in \mathbb{R}$.
Using Heath-Brown's identity, the estimate of the exponential sum involving
the M\"{o}bius function in short intervals
\begin{equation} \label{eqn: exponential sums in short intervals}
  S_k(x,y;\alpha) = \sum_{x < n \leq x+y} \mu(n) e\left(n^k \alpha\right)
\end{equation}
was first studied by Zhan \cite{zhan1991davenport}.
For the case $k=1$, Zhan \cite{zhan1991davenport} gave an upper bound of
the form $y (\log y)^{-A}$ for any $A>0$, which holds for $y \geq x^{2/3 + \varepsilon}$,
and then refined by Zhan \cite{zhan1991representation} to $y \geq x^{5/8 + \varepsilon}$.
For the case $k=2$, Liu and Zhan \cite{liu1999estimation} first established a nontrivial
estimate of $S_2(x,y;\alpha)$ for $y \geq x^{11/16 + \varepsilon}$ and all $\alpha \in \mathbb{R}$.
In \cite{lv2007exponential}, L\"u and Lao improved this result to $y \geq x^{2/3 + \varepsilon}$
which was as good as what was previously derived from the Generalized Riemann Hypothesis in \cite{liu1999estimation}.
For the case $k\geq 3$, the result for all $\alpha$ was first given
by Liu and Zhan \cite{liu1996estimation}, and recently Huang and Wang \cite{huang2015exponential}
gave an improvement by combining with the method of Kumchev.

Following \cite{sarnak2011three}, we say two sequences $(a_n)$ and $(b_n)$
of complex numbers are \emph{asymptotically orthogonal} (in short, ``orthogonal'') if
\begin{equation}
  \sum_{n\leq N} a_n b_n = o\left( \left(\sum_{n\leq N} |a_n|^2 \right)^{1/2} \left(\sum_{n\leq N} |b_n|^2 \right)^{1/2}  \right)
\end{equation}
as $N\rightarrow \infty$; and \emph{strongly asymptotically orthogonal}
(in short, ``strongly orthogonal'') if
\begin{equation}
  \sum_{n\leq N} a_n b_n = O_{A} \left( (\log N)^{-A} \sum_{n\leq N} |a_n b_n|  \right)
\end{equation}
for every $A > 0$, uniformly for $N \geq 2$.
The \emph{M\"obius randomness law} (see \cite[\S 13.1]{iwaniec2004analytic})
asserts that the sequence $(\mu(n))$
should be orthogonal to any ``\emph{reasonable}'' sequence.
Sarnak has recently posed a more precise conjecture in this direction
and we refer the reader to \cite{sarnak2010mobius}, \cite{sarnak2011three}
and \cite{liu2013mobius} for recent developments on this theme.
In particular, Sarnak \cite[Conjecture 4]{sarnak2011three} proposed to replace the condition
``\emph{reasonable}'' by ``\emph{bounded with zero topological entropy}''.
The bound (\ref{eqn: davenport}) shows that the two sequences $(\mu(n))$ and $(e(n^k \alpha))$ are strongly orthogonal.
Similarly, we can say two sequences $(a_n)$ and $(b_n)$ of complex numbers
are \emph{orthogonal in short intervals of exponent $\Delta$} if
\begin{equation}
  \sum_{x<n\leq x+y} a_n b_n = o\left( \left(\sum_{x<n\leq x+y} |a_n|^2 \right)^{1/2} \left(\sum_{x<n\leq x+y} |b_n|^2 \right)^{1/2}  \right)
\end{equation}
as $x\rightarrow \infty$, for $y\geq x^{1-\Delta+\varepsilon}$;
and \emph{strongly orthogonal in short intervals of exponent $\Delta$} if
\begin{equation}
  \sum_{x<n\leq x+y} a_n b_n = O_{A} \left( (\log x)^{-A} \sum_{x<n\leq x+y} |a_n b_n|  \right)
\end{equation}
for every $A > 0$, uniformly for $2\leq y \leq x$ and $y\geq x^{1-\Delta+\varepsilon}$.

In this paper, the question we seek to answer is how large the exponent $\Delta_k$ can be for two sequences $(\mu(n))$ and $(e(n^k \alpha))$ uniformly for all $\alpha\in [0,1]$ in the general case $k\geq 3$.
That is, we deal with $S_k(x,y;\alpha)$ for all $\alpha\in [0,1]$ and $k\geq 3$.
We say that the exponent $\Delta_k$ is \emph{admissible} if $(\mu(n))$ and $(e(n^k \alpha))$ are strongly orthogonal in short intervals of exponent $\Delta_k$ for all $\alpha\in [0,1]$.
So far, in \cite{huang2015exponential} the authors show that one has the admissible exponent
\begin{equation}
    \Delta_k = \left\{\begin{array}{ll}
      \frac{1}{5},    &  \textrm{if }\ k = 3; \\
      \frac{1}{2k},   &  \textrm{if }\ k \geq 4,
    \end{array}\right.
\end{equation}
which will be very small with the increase of $k$.
The main result of this paper shows that there are large admissible exponents $\Delta_k$ for all $k\geq 3$ being fixed.

\begin{theorem} \label{thm: exponential sums}
  Let $k\geq 3$. The exponent $\Delta_k = 1/4$ is admissible. That is, if $y = x^\theta$ with $3/4 < \theta \leq 1$, then for any $A>0$, we have
  \begin{equation*}
    S_k(x,y;\alpha) \ll y (\log y)^{-A},
  \end{equation*}
  uniformly for $\alpha \in (-\infty, +\infty)$.
\end{theorem}

\begin{remark}
  In contrast to the admissible exponents derived in the previous work cited above, this exponent is bounded away from zero as $k\rightarrow\infty$.

  An estimate for
  \begin{equation}
    \sum_{x < n \leq x+y} \Lambda(n) e\left(n^k \alpha\right)
  \end{equation}
  can be established by the same methods in this paper, where $\Lambda(n)$ is the von Mongoldt function. And then combined with the Hardy--Littlewood circle method, this enables us to give some short interval variants of Hua's theorems in additive number theory \cite{hua1965additive}, such as \cite[Theorems 2 and 3]{liu1996estimation} and \cite[Theorems 2 and 3]{huang2014exponential}.
\end{remark}

\noindent
{\bf Notation.}
Throughout the paper, the letter $\varepsilon$ denotes a sufficiently small positive real number, while $c$ without subscript stands for an absolute positive constant; both of them may be different at each occurrence. For example, we may write
\[
  (\log x)^c (\log x)^c \ll (\log x)^c, \quad  x^\varepsilon \ll y^\varepsilon.
\]
Any statement in which $\varepsilon$ occurs holds for each positive $\varepsilon$, and any implied constant in such a statement is allowed to depend on $\varepsilon$.
The letter $p$, with or without subscripts, is reserved for prime numbers.
In addition, as usual, $e(z)$ denotes $\exp(2\pi iz)$.
We write $(a, b) = {\rm gcd}(a, b$), and we use $m \sim M$ as an abbreviation for the condition $M < m \leq 2M$.

\section{Outline of our method}

Let
\begin{equation} \label{eqn: PQR}
  P = L^{c_1}, \quad
  Q = x^{k-2} y^2 / P, \quad
  R = x^{k-1} y,
\end{equation}
where here and in the sequel $L$ stands for $\log x$, and the letter $c$ with or
without subscripts denotes positive constants which depend at most on $k$ and $A$,
fixed in advance.
Write $\alpha$ in the form
\begin{equation} \label{eqn: Dirichlet's theorem on rational approximations}
  \alpha = \frac{a}{q} + \lambda, \quad (a,q) = 1.
\end{equation}
To estimate $S_k(x,y,\alpha)$ for $\alpha$ in $[0,1]$, we divide $[0,1]$ into
three subsets according to the idea due to Pan \cite{pan1959some}.
Let $P,\ Q$ and $R$ be defined as in (\ref{eqn: PQR}). It follows from Dirichlet's lemma on rational approximations that every $\alpha \in [0,1]$ can be written as (\ref{eqn: Dirichlet's theorem on rational approximations}), with $q,\lambda$ satisfying one of the following three conditions:
\begin{eqnarray*}
  (a)  &&  q \leq P,\ |\lambda|\leq \frac{1}{R}; \\
  (b)  &&  q \leq P,\ \frac{1}{R} < |\lambda| \leq \frac{1}{qQ}; \\
  (c)  &&  P < q \leq Q,\ |\lambda|\leq \frac{1}{qQ}.
\end{eqnarray*}
Denote by $\mathcal{A},\ \mathcal{B}$ and $\mathcal{C}$ the three subsets of $\alpha$ satisfying (a), (b) and (c) respectively. Then $[0,1]$ is the disjoint union of $\mathcal{A},\ \mathcal{B}$ and $\mathcal{C}$.

In \S \ref{sec: A} and \S \ref{sec: B} we employ analytic methods to deal with the case $\alpha\in \mathcal{A} \cup \mathcal{B}$, and obtain the following.
\begin{proposition} \label{prop: A}
  Let $k\geq 3$ and $y = x^\theta$ with $7/12 < \theta \leq 1$. Then for any $c_1$, we have
  \[
    S_k(x,y;\alpha) \ll y L^{-A}
  \]
  holds uniformly for $\alpha \in \mathcal{A}$.
\end{proposition}
\begin{proposition} \label{prop: B}
  Let $k\geq 3$, and $y = x^\theta$ with $2/3 < \theta < 1$. Then for any $c_1$, we have
  \[
    S_k(x,y;\alpha) \ll y L^{-A}
  \]
  holds uniformly for $\alpha \in \mathcal{B}$.
\end{proposition}

To prove the propositions, we appeal to zero-density estimate of $L$-functions
in short intervals, see Zhan \cite{zhan1991davenport}.
Since there is no explicit formula for $\sum_{n\leq u} \mu(n)\chi(n)$
in terms of zeros of the $L$-function with $\chi$ being a primitive character modulo $l$ and $l\leq P$,
we may use the so-called \emph{Hooley--Huxley contour} method as in \cite{ramachandra1976some}.
For $\alpha\in \mathcal{A}$, we can use the partial summation formula
to handle the exponential function $e(n^k\lambda)$. But for $\alpha \in \mathcal{B}$,
we must employ the exponential integral to deal with the exponential function.

In \S \ref{sec: C} we follow Kumchev's approach to handle the case $\alpha\in \mathcal{C}$, and get the following.
\begin{proposition} \label{prop: C}
  Let $k\geq 3$ and $y = x^\theta$ with $3/4 < \theta \leq 1$. Then there exists $c_1>0$ such that the estimate
  \[
    S_k(x,y;\alpha) \ll y L^{-A}
  \]
  holds uniformly for $\alpha \in \mathcal{C}$.
\end{proposition}

In proving the above proposition, we make use of the results in Daemen \cite{daemen2010asymptotic} (see Lemma \ref{lemma: exponential sums in short intervals}) and apply Kumchev's method to estimate the exponential sums of type I and type II:
\begin{equation}
  \sum_{m\sim M} a(m) \sum_{ x<mn\leq x+y} e\left((mn)^k \alpha\right),
\end{equation}
\begin{equation}
  \sum_{m \sim M} a(m) \sum_{ x< mn \leq x+y} b(n) e\left((mn)^k \alpha\right)
\end{equation}
respectively, then appeal to Vaughan's identity, where $a(m)\ll \tau^c(m)$, $b(n)\ll \tau^c(n)$, and $\tau(n)$ is the divisor function.

It is easily seen that Theorem \ref{thm: exponential sums} follows from Propositions \ref{prop: A}, \ref{prop: B} and \ref{prop: C}.

\section{Preliminaries} \label{eqn: preliminaries}

%

\begin{lemma} \label{lemma: S_k(chi)}
  Let $k\geq 3$ and $\eta_k = 1 - \frac{1}{k}$. Write $\alpha=\frac{a}{q}+\lambda$ with $(a,q)=1$.
  Then for any $\varepsilon > 0$, we have
  \begin{equation*}
    S_k(x,y;\alpha) \ll q^{\eta_k+\varepsilon} \sum\limits_{d|q} \max_{\chi_{q/d}}
    \bigg| \sum\limits_{x/d < m \leq (x+y)/d \atop (m,q)=1}\mu(m)\chi(m)e(m^kd^k\lambda) \bigg|.
  \end{equation*}
  Here the implied constant is absolute.
\end{lemma}

\begin{proof}
  It is analogous to the proof of \cite[Lemma 2]{liu1996exponential}.
\end{proof}

%

\begin{lemma} \label{lemma: zero density for alpha near 1}
  Let $N(\alpha,T,\chi)$ be the number of zeros of $L(s,\chi)$ in
  $$
    \Re(s) \geq \alpha \ \textrm{and} \ |\Im(s)| \leq T,
  $$
  where $\chi = \chi_l$ is a primitive character modulo $l$.
  For $\frac{1}{2} \leq \alpha \leq 1$, $T\geq 2$ and $l\geq 1$, we have
  $$
    N(\alpha,T,\chi) \ll  (lT)^{167(1-\alpha)^{3/2}} (\log lT)^{17}.
  $$
\end{lemma}

\begin{proof}
  See \cite[Corollary 12.5]{montgomery1971topics}.
\end{proof}

\begin{lemma} \label{lemma: zero density in short intervals}
  Use the notation in Lemma \ref{lemma: zero density for alpha near 1}. Let
  $$ N(\alpha,T,H,\chi)  =  N(\alpha,T+H,\chi)-N(\alpha,T,\chi).$$
  Then for $\frac{1}{2}\leq \alpha \leq 1$, $T^{\frac{35}{108}+\varepsilon} \leq H \leq T$ and $l \geq 1$, we have
  $$
    N(\alpha,T,H,\chi)\ll (lH)^{\frac{8}{3}(1-\alpha)}(\log lH)^{216}.
  $$
\end{lemma}

\begin{proof}
  See \cite[Theorem 3]{zhan1992mean}.
\end{proof}

\begin{lemma} \label{lemma: zero free region}
  For $l\geq 1$, $L(s,\chi)$ has no zeros in
  $$
    \sigma \geq 1-\frac{c_0}{\log l + (\log(T+2))^{\frac{4}{5}}}  \quad \textrm{and} \quad   |t|\leq T,
  $$
  where $c_0>0$ is a constant, except the only exceptional zero $\tilde{\beta}$. And for $l\leq (\log T)^c$ no such exceptional zero exists.
\end{lemma}

\begin{proof}
  See \cite[Satz 6.2]{prachar1957primzahlverteilung}.
\end{proof}

\begin{lemma}[Vaughan's identity] \label{lemma: Vaughan's identity}
  Let $U,V \geq 1$. Then for any $n > \max\{U,V\}$, we have
  \begin{equation}
    \mu(n) = - \sum\limits_{\substack{lmd=n \\ 1\leq d\leq V \\  1\leq m\leq U}} \mu(d) \mu(m) + \sum\limits_{\substack{lmd=n \\ d>V \\ m>U}} \mu(d) \mu(m).
  \end{equation}
\end{lemma}

\begin{proof}
  See \cite[Proposition 13.5]{iwaniec2004analytic}.
\end{proof}

\begin{lemma} \label{lemma: shiu}
  Suppose that $1\leq N < N'<  2x$, $N'- N > x^{\varepsilon}d$ and $(c,d) = 1$. Then for $j,\nu \geq 1$, we have
  \[\sum_{\substack{N\leq n\leq N'\\ n\equiv c~(\mathrm{mod}~d)}}\tau_j(n)^{\nu} \ll \frac{N'-N}{\varphi(d)} (\log N)^{j^{\nu}-1},\]
  the implied constant depending on $\varepsilon, j$, and $\nu$ at most.
\end{lemma}

\begin{proof}
  See \cite[Theorem 1]{shiu1980brun}.
\end{proof}

When $k\geq 3$, we define the multiplicative function $w_k(q)$ by
\begin{equation*}
    w_k\left(p^{ku+v}\right) = \left\{ \begin{array}{ll}
      k p^{-u-1/2},   & \textrm{if } u\geq0,\ v=1, \\
      p^{-u-1},       & \textrm{if } u\geq0,\ v=2,\ldots,k.
    \end{array}\right.
\end{equation*}
By the argument of \cite[Theorem 4.2]{vaughan1997hardy}, we have
\begin{equation} \label{eqn: S(q,a)}
  S(q,a) = \sum_{1\leq x\leq q} e(ax^k/q) \ll q w_k(q) \ll q^{1-1/k}
\end{equation}
whenever $k\geq 3$ and $(a,q)=1$.
We also need several estimates for sums involving the function $w_k(q)$. We list those in the following lemma.

\begin{lemma} \label{lemma: some inequalities}
  Let $w_k(q)$ be the multiplicative function defined above. Then the following inequalities hold for any fixed $\varepsilon > 0$:
  \begin{equation}
    \sum_{n\sim N} \tau^{c}(n) w_k\left( \frac{q}{(q,n^j)} \right) \ll  q^\varepsilon (\log N)^C w_k(q) N  \quad  (1\leq j \leq k),
  \end{equation}
  where $\tau(q)$ is the divisor function and $C$ is a constant depending on $c$;
  \begin{equation}
    \sum_{\substack{n\sim N \\ (n,h)=1}} \tau^c(n) \tau^c(n+h) w_k\left( \frac{q}{(q,R(n,h))} \right) \ll  q^\varepsilon (\log N)^C w_k(q) N + q^\varepsilon,
  \end{equation}
  where $R(n,h) = \left( (n+h)^k - n^k \right)/h$.
\end{lemma}

\begin{proof}
  See Lemma 2.3 and inequality (3.11) in Kawada and Wooley \cite{kawada2001waring} and combine with the result in Lemma \ref{lemma: shiu}.
\end{proof}

\begin{lemma} \label{lemma: exponential sums in short intervals}
  Let $k\geq 3$ be an integer and $\gamma\geq 3$ be a real number. Let $0<\rho\leq \sigma_k/\gamma$, where $\sigma_k = \frac{1}{2k(k-1)}$. Suppose that $y\leq x,$ and $y \geq x^{\frac{\gamma}{2\gamma-\sigma_k-1}}$. Then either
  \begin{equation} \label{eqn: exponential sums in short intervals, minor arcs}
    \sum_{x < n \leq x+y} e\left(n^k \alpha\right) \ll y^{1-\rho +\varepsilon},
  \end{equation}
  or there exist integers $a$ and $q$ such that
  \begin{equation} \label{eqn: exponential sums in short intervals, conditions}
    1\leq q \leq y^{k\rho}, \quad (a,q)=1, \quad |q\alpha - a| \leq x^{1-k} y^{k\rho -1},
  \end{equation}
  and
  \begin{equation} \label{eqn: exponential sums in short intervals, major arcs}
    \sum_{x < n \leq x+y} e\left(n^k \alpha\right)  \ll   y^{1-\rho +\varepsilon} + \frac{w_k(q) y}{1 + yx^{k-1}|\alpha - a/q|}.
  \end{equation}
\end{lemma}

\begin{proof}
  Take
  \begin{equation*}
    P_0 = y^{1/\gamma}, \quad  Q_0 = x^{k-2}y^2/P_0.
  \end{equation*}
  By Dirichlet's lemma on rational approximations, there exists integers $a$ and $q$ with
  \begin{equation} \label{eqn: conditions of a and q}
    1\leq q\leq Q_0, \quad  (a,q)=1, \quad  |q\alpha-a|\leq 1/Q_0.
  \end{equation}
  When $q>P_0$, we rewrite the sum on the left of (\ref{eqn: exponential sums in short intervals, minor arcs}) as
  \begin{equation*}
    \sum_{1\leq n\leq z} e(\alpha_k n^k + \alpha_{k-1} n^{k-1} + \cdots + \alpha_0),
  \end{equation*}
  where $z\leq y$ and $\alpha_j = \binom{k}{j} \alpha u^{k-j}$, with $u$ a fixed integer.
  Hence, it follows from the argument underlying the proof of \cite[equation (3.5)]{daemen2010asymptotic} and \cite[equation (4.23)]{wei2014sums} that
  \begin{equation}
    \sum_{x < n \leq x+y} e\left(n^k \alpha\right) \ll y P_0^{-\frac{1}{2k(k-1)}+\varepsilon} \ll y^{1-\rho +\varepsilon}.
  \end{equation}
  When $q\leq P_0$, from (\ref{eqn: S(q,a)}) and \cite[equations (5.1)-(5.5) and \S 6]{daemen2010asymptotic}, we deduce
  \begin{equation*}
    \sum_{1\leq n\leq y} e(\alpha_k n^k + \alpha_{k-1} n^{k-1} + \cdots + \alpha_0) \ll \frac{w_k(q) y}{1 + yx^{k-1}|\alpha - a/q|} + \Delta,
  \end{equation*}
  where
  \begin{equation*}
    \Delta \ll P_0^{1/2+\varepsilon} \left( 1+ \frac{P_0 x^k}{x^{k-2} y^2} \right)^{1/2} \ll P_0^{1+\varepsilon} x/y \ll y^{1-\rho +\varepsilon},
  \end{equation*}
  provided that $y \geq x^{\frac{\gamma}{2\gamma-\sigma_k-1}}$. Thus, at least one of (\ref{eqn: exponential sums in short intervals, minor arcs}) and (\ref{eqn: exponential sums in short intervals, major arcs}) holds.
  The lemma follows on noting that when conditions (\ref{eqn: exponential sums in short intervals, conditions}) fail, inequality (\ref{eqn: exponential sums in short intervals, minor arcs}) follows from (\ref{eqn: exponential sums in short intervals, major arcs}).
\end{proof}

\section{The case $\alpha \in \mathcal{A}$} \label{sec: A}

Let
\begin{equation}
  S_k(\chi)=\sum\limits_{x_1<m \leq x_1+y_1 \atop (m,q)=1}\mu(m)\chi(m)e(m^kd^k\lambda),
\end{equation}
where $\chi=\chi_{l}$ is a primitive character, $ld|q$, $x_1=x/d$ and $y_1=y/d.$
By Lemma \ref{lemma: S_k(chi)}, in order to prove Propositions \ref{prop: A} and \ref{prop: B} it is sufficient to establish that
\begin{equation} \label{eqn: S_k(chi)}
  S_k(\chi) \ll q^{-1} y L^{-A}.
\end{equation}

It follows from the main theorem in Ramachandra \cite{ramachandra1976some} that
\begin{equation} \label{eqn: mobius sum in short intervals}
  \sum_{x_1<m\leq u \atop (m,q)=1} \mu(m)\chi(m)  \ll  (u-x_1) \exp(-c(\log x_1)^{1/6})
\end{equation}
holds for $x_1^{7/12+\varepsilon} \leq u-x_1 \leq y_1$.
Hence for $\alpha\in\mathcal{A}$, (\ref{eqn: S_k(chi)}) can be proved by the partial summation formula.
Recall that $|\lambda| \leq \frac{1}{R}$, we have
\begin{equation*}
  \begin{split}
    S_k(\chi) & = \int_{x_1}^{x_1+y_1} e(\lambda d^k u^k)  d \left( \sum\limits_{x_1<m \leq u \atop (m,q)=1} \mu(m)\chi(m) \right) \\
              & \ll y_1 \exp(-c(\log x_1)^{1/6}) + \int_{x_1}^{x_1+x_1^{7/12+\varepsilon}} (u-x_1) |\lambda| u^{k-1} du \\
              & \qquad + \int_{x_1+x_1^{7/12+\varepsilon}}^{x_1+y_1} (u-x_1)\exp(-c(\log x_1)^{1/6})  |\lambda| u^{k-1} du \\
              & \ll x^{7/12+\varepsilon} + y \exp(-c' L^{1/6}) \ll q^{-1} y L^{-A}.
  \end{split}
\end{equation*}
This proves Proposition \ref{prop: A}.

\section{The case $\alpha \in \mathcal{B}$} \label{sec: B}

Recall that for $\alpha \in  \mathcal{B}$, we have $\frac{1}{R} < |\lambda| \leq \frac{1}{qQ}$.
We start with Perron's summation formula (see \cite[\S 5, Theorem 1]{karatsuba1993basic}). For $x_1<u<2x_1$
\begin{equation} \label{eqn: Perron's formula}
  \sum_{x_1<m\leq u \atop (m,q)=1} \mu(m)\chi(m) = \frac{1}{2\pi i} \int_{b_0-iT}^{b_0+iT} F(s,\chi)\frac{u^s-x^s_1}{s}ds + O\left(\frac{x_1 L}{T}\right),
\end{equation}
where $b_0 = 1+\frac{1}{L}$ and
\begin{equation}
  F(s,\chi)=\sum\limits_{m=1 \atop (m,q)=1}^{\infty} \frac{\mu(m)\chi(m)}{m^{s}}, \quad  \textrm{Re}(s)>1.
\end{equation}
Let $H(s,\chi)=\prod\limits_{p|q} \left(1-\frac{\chi(p)}{p^s}\right)^{-1}$, 
then $F(s,\chi)=H(s,\chi)/L(s,\chi).$
Moreover, let $\rho(q)$ be defined by
\begin{equation}
  \rho(q) = \prod_{p|q} \left( 1+ \frac{1}{\sqrt{p}} \right).
\end{equation}
Then we have
\begin{equation*}
  F(\sigma+it,\chi)  \ll \frac{1}{|L(\sigma+it,\chi)|} \prod\limits_{p|q} \left( 1-\frac{1}{\sqrt{p}} \right)^{-1}  \ll  \frac{L\rho(q)}{|L(\sigma+it,\chi)|}
\end{equation*}
for $\textrm{Re}(s)=\sigma> \frac{1}{2}$.

Let $M$ be the so-called \emph{Hooley--Huxley contour} as described
by Ramachandra \cite{ramachandra1976some}. Briefly speaking, we take the rectangle
\begin{equation*}
  \frac{1}{2}\leq \sigma \leq 1, \quad |t|\leq T+2000(\log T)^2,
\end{equation*}
and divide it into equal rectangles of hight $400(\log T)^2$
(the real line cuts one of these rectangles into two equal parts,
we denote this rectangle by $R_0$).
Let $R_n \ (n=-n_1, \cdots, n_1)$ be all these rectangles.
In $R_n$, we fix a new right-hand side and obtain a new rectangle as follows.
Consider $R_{n-1},R_n$, and $R_{n+1}$ whenever all of the three are defined.
Pick out a zero of $L(s,\chi)$ in $R_{n-1}\cup R_n \cup R_{n+1}$ 
with the greatest real part $\beta_n$ and $\textrm{Re}(s)=\beta_n$ is the new right-hand side of $R_n$.
Now we join all the right edges of the new rectangles by horizontal lines. These form the contour $M'$.

The Hooley--Huxley contour is obtained by making the following changes on $M'$.
Let $a$, $b$, and $\vartheta$ be positive constant to be chosen later, satisfying $0<\vartheta <1$, $a$ should be small and $b$ should be close to 1. If $\beta_n <\vartheta$, then in place of $\beta_n$ we take $\beta'_n=\beta_n+3a(1-\beta_n)$. If $\beta_n \geq \vartheta$, then $\beta_n$ is replaced by $\beta'_n=\beta_n+b(1-\beta_n)$. These form the Hooley--Huxley contour.

Now we join the points $b_0 \pm iT$ to $M$ by horizontal lines $H_1$ and $H_2$.
The parameter $T$ will be chosen as a suitable power of $x$. Since
$$
  \frac{1}{|L(s,\chi)|}\ll T^\varepsilon
$$
for $s$ on $H_1$ and $H_2$ as shown in \cite{ramachandra1976some},
shifting the integral line in (\ref{eqn: Perron's formula}) to $M$ we obtain
\begin{equation} \label{eqn: Perron M}
  \sum_{x_1<m\leq u \atop (m,q)=1}\mu(m)\chi(m)
  = \frac{1}{2\pi i} \int_{M} F(s,\chi)\frac{u^s-x^s_1}{s}ds + O\left(\frac{x_1 L}{T^{1-\varepsilon}}\right).
\end{equation}
Therefore
\begin{equation}
  \begin{split}
    S_k(\chi) & = \int_{x_1}^{x_1+y_1}e(\lambda d^ku^k)d\sum_{x_1< m \leq u \atop (m,q)=1}\mu(m)\chi(m)  \\
              & = \frac{1}{2\pi i} \int_{x_1}^{x_1+y_1} e(\lambda d^ku^k)du  \int_M F(s,\chi)u^{s-1}ds
                   + O\left( \frac{1+|\lambda|x^{k-1}y}{T^{1-\varepsilon}} x_1 L \right).
  \end{split}
\end{equation}
Taking
\begin{equation} \label{eqn: T}
  T^{1-\varepsilon} = (1 + |\lambda| x^{k-1}y) q x y^{-1} L^{A+1},
\end{equation}
we have
\begin{equation}
    S_k(\chi) = \frac{1}{2\pi i} \int_M F(s,\chi)ds \int_{x_1}^{x_1+y_1} u^{s-1} e(\lambda d^ku^k) du +  O\left( q^{-1} y L^{-A} \right).
\end{equation}

Let
\begin{equation*}
  \begin{split}
    I & := \int_{x_1}^{x_1+y_1} u^{s-1} e(\lambda d^ku^k) du\\
      & = \int_{x_1}^{x_1+y_1} u^{\sigma-1} e\left(\lambda d^ku^k+\frac{t}{2k\pi}\log u^k\right) du\\
      & = \frac{1}{k} \int_{x_1^k}^{(x_1+y_1)^k} v^{\frac{\sigma}{k}-1} e\left(\lambda d^k v + \frac{t}{2k\pi}\log v \right) dv.
  \end{split}
\end{equation*}
Let $\mathcal{V}$ denote the interval $[x_1^k,(x_1+y_1)^k]$, and
\begin{equation*}
  f(v) = \lambda d^kv + \frac{t}{2k\pi} \log v, \quad  v\in \mathcal{V}.
\end{equation*}
Then we have
\begin{equation*}
  \begin{split}
    f'(v)  & = \lambda d^k + \frac{t}{2k\pi v} \gg \frac{\min\limits_{v\in \mathcal{V}} |t + 2k\pi \lambda d^k v|}{x_1^k}, \\
    f''(v) & = -\frac{t}{2k\pi v^2} \gg \frac{|t|}{x_1^{2k}}.
  \end{split}
\end{equation*}
Hence, we have (see Titchmarsh \cite[Lemmas 4.3 and 4.4]{titchmarsh1986theory})
\begin{equation*}
  \begin{split}
    I & \ll x_1^{\sigma - k} \min\left( x_1^{k-1}y_1, \frac{x_1^{k}}{\min\limits_{v\in \mathcal{V}} |t + 2k\pi \lambda d^k v|}, \frac{x_1^k}{\sqrt{|t|}} \right) \\
      & = x_1^{\sigma -1} \min\left( y_1, \frac{x_1}{\min\limits_{v\in \mathcal{V}} |t + 2k\pi \lambda d^k v|}, \frac{x_1}{\sqrt{|t|}} \right).
  \end{split}
\end{equation*}
Therefore
\begin{equation}
  S_k(\chi) \ll  \int_M  \min\left( y_1, \frac{x_1}{\min\limits_{v\in \mathcal{V}} |t + 2k\pi \lambda d^k v|}, \frac{x_1}{\sqrt{|t|}} \right) x_1^{\sigma -1} |F(s,\chi)| |ds|  +   q^{-1} y L^{-A}.
\end{equation}

Take
\begin{equation} \label{eqn: H}
  H = \frac{x}{y} + 2^{k+2} k\pi |\lambda| x^{k-1}y.
\end{equation}
For $|t + 2k\pi \lambda x^k| \leq H$, since $\frac{1}{R} < |\lambda| \leq \frac{1}{qQ}$, we have
\begin{equation*}
  \min\left( y_1, \frac{x_1}{\sqrt{|t|}} \right) \ll \min\left( y_1, \frac{x_1}{\sqrt{|\lambda|x^k}} \right)
\end{equation*}
holds. And the inequality
\begin{equation*}
  |t + 2k\pi \lambda x^k| \geq jH, \quad  j\geq 1
\end{equation*}
ensures that, for $v\in \mathcal{V}$,
\begin{equation}
  |t + 2k\pi \lambda d^k v| \geq jH - 2k\pi |\lambda| ((x+y)^k - x^k) \geq \frac{1}{2} jH.
\end{equation}
Then we have
\begin{equation*}
  \begin{split}
        & \int_{M}  \min\left( y_1, \frac{x_1}{\min\limits_{v\in \mathcal{V}} |t + 2k\pi \lambda d^k v|}, \frac{x_1}{\sqrt{|t|}} \right) x_1^{\sigma -1} |F(s,\chi)| |ds| \\
    \ll & \int_{|t + 2k\pi \lambda x^k| \leq H}  \min\left( y_1, \frac{x_1}{\sqrt{|t|}} \right) x_1^{\sigma -1} |F(s,\chi)| |ds|  \\
        & + \sum_{\substack{j\geq 1 \\ jH\leq 2T}}  \int_{jH\leq |t + 2k\pi \lambda x^k| \leq (j+1)H}  \frac{x_1}{\min\limits_{v\in \mathcal{V}} |t + 2k\pi \lambda d^k v|} x_1^{\sigma -1} |F(s,\chi)| |ds| \\
    \ll & L \max_{|T_1|\leq 2T} \int_{T_1}^{T_1+H} \left( \min\left( y_1, \frac{x_1}{\sqrt{|\lambda|x^k}} \right) + \frac{x_1}{H} \right) x_1^{\sigma -1} |F(s,\chi)| |ds|.
  \end{split}
\end{equation*}
For $\alpha\in \mathcal{B}$, by (\ref{eqn: H}), it is a simple matter to show that
\begin{equation*}
  \min\left( y_1, \frac{x_1}{\sqrt{|\lambda|x^k}} \right) + \frac{x_1}{H}
  \ll y_1 + \frac{x_1}{H}
  \ll \sqrt{\frac{x_1y_1}{H}} L^{c_1} = \frac{1}{d} \sqrt{\frac{xy}{H}} L^{c_1}.
\end{equation*}
Let $M(H)$ denote the part of $M$ satisfying
\begin{equation*}
  T_1\leq \Im(s) \leq T_1+H, \ |T_1|\leq 2T.
\end{equation*}
Then
\begin{equation*}
\begin{split}
 S_k(\chi) & \ll L^{c_1+1} \sqrt{\frac{xy}{H}} \max_{|T_1|\leq 2T} \int_{M(H)} x^{\sigma-1} |F(s,\chi)||ds| +  q^{-1} y L^{-A}\\
           & \ll L^{c_1+2} \rho(q) \sqrt{\frac{xy}{H}} \max_{|T_1|\leq 2T} \int_{M(H)} x^{\sigma-1} |L(s,\chi)|^{-1} |ds| +  q^{-1} y L^{-A}.
\end{split}
\end{equation*}
Since $H\geq xy^{-1}$, we have $\sqrt{\frac{xy}{H}} \leq y$.
To prove Proposition \ref{prop: B}, now it is sufficient to show that for $|T_1|\leq 2T$,
\begin{equation} \label{eqn: int M(H)}
  \int_{M(H)} x^{\sigma-1} |L(s,\chi)|^{-1} |ds|  \ll  L^{-A-2c_1-2}.
\end{equation}

To prove (\ref{eqn: int M(H)}), we just follow the method of Ramachandra \cite{ramachandra1976some}.
It is shown in \cite[Lemma 5]{ramachandra1976some} that
\begin{eqnarray*}
  && |L(s,\chi)|^{-1} \ll T^\varepsilon, \ \textrm{if} \ s\in M(H) \ \textrm{and} \ \Re(s)\leq \vartheta+b(1-\vartheta), \\
  && |L(s,\chi)|^{-1} \ll \textrm{exp}((\log T)^{3(1-b)}), \ \textrm{if} \ s\in M(H) \ \textrm{and} \ \Re(s)> \vartheta+b(1-\vartheta).
\end{eqnarray*}
We divide the smallest vertical strip containing $M(H)$ into vertical strips of width $1/\log T$.
Consider the bits of $M(H)$, say $M(H,\sigma')$, in the vertical strip about the abscissa $\sigma'$.
Then by the construction of the Hooley--Huxley contour, we have 
$$
  \int_{M(H,\sigma')} |ds|  \ll  N(\sigma',T_1,H,\chi)(\log T)^{10},
$$
where $\sigma'$ is $\sigma+3a(1-\sigma)$ or $\sigma+b(1-\sigma)$ according as $\sigma' \leq \vartheta$ or $\sigma' > \vartheta$.
By the above discussion and Lemmas \ref{lemma: zero density for alpha near 1}-\ref{lemma: zero free region}, we obtain
\begin{equation*}
  \begin{split}
    & \int_{M(H)} x^{\sigma-1} |L(s,\chi)|^{-1} |ds|    \\
  = & \int_{M(H) \atop \sigma'<\vartheta}x^{\sigma'-1} |L(s,\chi)|^{-1} |ds| + \int_{M(H)\atop \vartheta\leq \sigma' \leq \vartheta+b(1-\vartheta)}x^{\sigma'-1} |L(s,\chi)|^{-1} |ds|   \\
    & + \int_{M(H)\atop \sigma' > \vartheta+b(1-\vartheta)} x^{\sigma'-1} |L(s,\chi)|^{-1} |ds|   \qquad  (s=\sigma'+it)  \\
  \ll &  T^\varepsilon \bigg(\frac{H^{\frac{8}{3}(1-3a)^{-1}}}{x}\bigg)^{(1-\vartheta)} + T^\varepsilon \bigg(\frac{T^{167(1-\vartheta)^{\frac{1}{2}}(1-b)^{-\frac{3}{2}}}}{x}\bigg)^{(1-b)(1-\vartheta)}  \\
    & + \exp\left((\log T)^{3(1-b)}\right) \bigg(\frac{T^{167(1-\vartheta)^{\frac{1}{2}}(1-b)^{-\frac{3}{2}}}}{x}\bigg)^{c_0(\log T)^{-4/5}}
  \end{split}
\end{equation*}
provided $a$, $b$, and $\vartheta$ satisfy
\begin{equation}\label{eqn: a}
  H^{\frac{8}{3}\left(1-3a\right)^{-1}}  \leq  x^{1-\varepsilon},
\end{equation}
\begin{equation}\label{eqn: vartheta}
  T^{400(1-\vartheta)^{\frac{1}{2}}(1-b)^{-\frac{3}{2}}}  \leq  x.
\end{equation}
In fact, we may first choose $a$ such that
$$
  \frac{8}{3} \left(\frac{1}{3}+\varepsilon\right) \frac{1}{1-3a}  <  1-\varepsilon \ \quad   (H\leq x^{\frac{1}{3}+\varepsilon}),
$$
$b$ such that $3(1-b)=\frac{1}{100}$ and then $\vartheta$ such that (\ref{eqn: vartheta}) holds.
Hence
$$
  \int_{M(H)} x^{\sigma-1} |L(s,\chi)|^{-1} |ds|  \ll  \exp(-c'_0L^{\frac{1}{6}})\ \quad  (c'_0>0)
$$
and (\ref{eqn: int M(H)}) follows. Thus, we prove Proposition \ref{prop: B}.

\section{The case $\alpha \in \mathcal{C}$} \label{sec: C}

\subsection{Type I estimate}
Recall that
\begin{equation*}
  y = x^\theta, \quad  L = \log x.
\end{equation*}
The following lemma treats the exponential sums of type I which is an improvement of \cite[Lemma 8]{huang2015exponential}.

\begin{lemma} \label{lemma: Type I}
  Let $k\geq 3$ be an integer and $\gamma\geq 3$ be a real number. Let $0 < \rho < \sigma_k /(2\gamma)$, with $\sigma_k= \frac{1}{2k(k-1)}$.
  Suppose that $\alpha \in \mathcal{C}$ and $a(m) \ll \tau^c(m)$. Define
  \begin{equation*}
    \mathcal{T}_1 = \sum_{m    \sim M}  a(m) \sum_{x<mn\leq x+y} e\left((mn)^k \alpha\right).
  \end{equation*}
  Then for any $A>0$, we have
  \begin{equation*}
    \mathcal{T}_1 \ll y L^{-A}, 
  \end{equation*}
  provided that
  \begin{equation} \label{eqn: conditions of M. type I}
    M \ll y \left( \frac{y}{x} \right)^{\frac{\gamma}{\gamma-\sigma_k-1}}, \quad
    M \ll y x^{-\gamma\rho/\sigma_k}, \quad
    M^{2k} \ll y x^{k-1-2k\rho},
  \end{equation}
  and
  \begin{equation} \label{eqn: c_1. I}
    c_1 > (k+1)(A+C),
  \end{equation}
  where $C$ is a constant depending on $c$.
\end{lemma}

\begin{proof}
  Set \[  S_m = \sum_{X < n \leq X+Y} e\left( m^k n^k \alpha \right), \]
  where $X = x/m, Y =y/m$ with $m\sim M$.
  Define $\nu$ by $Y^\nu = x^\rho L^{-1}$. Note that, by (\ref{eqn: conditions of M. type I}), we have
  $$ \nu < \sigma_k/\gamma. $$
  We denote by $\mathcal{M}$ the set of integers $m \sim M$, for which there exist integers $b_1$ and $r_1$ with
  \begin{equation} \label{eqn: conditions of r_1 and b_1}
    1 \leq r_1 \leq Y^{k\nu},\quad (b_1,r_1)=1,\quad  |r_1 m^k \alpha - b_1| \leq X^{1-k} Y^{k\nu -1}.
  \end{equation}
  By (\ref{eqn: conditions of M. type I}), we have $Y \gg X^{\gamma/(2\gamma-\sigma_k-1)}$.
  We apply Lemma \ref{lemma: exponential sums in short intervals} to the summation over $n$ and get
  \begin{equation*}
    S_m \ll  Y^{1-\nu+\varepsilon} +  \frac{w_k(r_1) Y}{1+ YX^{k-1}|m^k\alpha- b_1/r_1|},
  \end{equation*}
  for $m \in \mathcal{M}$.
  Consequently,
  \begin{equation*}
    \begin{split}
      \mathcal{T}_1 &\ll  \sum_{m\sim M} a(m) Y^{1-\nu+\varepsilon}  + \sum_{m\in \mathcal{M}} \frac{a(m) w_k(r_1) Y}{1+ YX^{k-1}|m^k\alpha- b_1/r_1|}\\
                    &\ll  x^{\theta-\rho +\varepsilon} + T_1(\alpha),
     \end{split}
  \end{equation*}
  where
  \[
    T_1(\alpha) = \sum_{m\in \mathcal{M}} \frac{a(m) w_k(r_1) Y}{1+ YX^{k-1}|m^k\alpha- b_1/r_1|}.
  \]
  We apply Dirichlet's lemma on rational approximations to find integers $b$ and $r$ with
  \begin{equation} \label{eqn: conditions of r and b}
    1\leq r \leq x^{-k\rho} Y X^{k-1}, \quad (b,r)=1, \quad  |r\alpha -b| \leq x^{k\rho} Y^{-1} X^{1-k}.
  \end{equation}
  By (\ref{eqn: conditions of M. type I}), (\ref{eqn: conditions of r_1 and b_1}) and (\ref{eqn: conditions of r and b}), we have
  \begin{eqnarray*}
    |b_1 r - bm^k r_1| &=& |r(b_1-r_1m^k\alpha) + r_1 m^k (r\alpha - b)| \\
     &\leq& x^{-k\rho} Y X^{k-1}  X^{1-k} Y^{k\nu -1} + Y^{k\nu} (2M)^k x^{k\rho} Y^{-1} X^{1-k} \\
     &\ll& L^{-k} +  M^{2k} L^{-k}  x^{2k\rho-k+1} y^{-1} \ll L^{-k} < 1,
  \end{eqnarray*}
  whence
  \[
    \frac{b_1}{r_1} = \frac{m^k b}{r}, \quad  r_1 = \frac{r}{(r,m^k)}.
  \]
  Thus, by Lemma \ref{lemma: some inequalities}, we have
  \begin{eqnarray*}
    T_1(\alpha) 
     &\ll&  \frac{yM^{-1}}{1+ yx^{k-1}|\alpha- b/r|}  \sum_{m\sim M} \tau^c(m) w_k\left(\frac{r}{(r,m^k)}\right) \\
     &\ll&  \frac{ w_k(r) r^\varepsilon L^C y}{1+ yx^{k-1}|\alpha- b/r|}.
  \end{eqnarray*}
  Recall that $b$ and $r$ satisfy the conditions (\ref{eqn: conditions of r and b}). We now consider three cases depending on the sizes of $r$ and $|r\alpha -b|$.
  \begin{description}
    \item[\it Case 1] If $r > L^{(k+1)(A+C)} $, then $T_1(\alpha) \ll yL^{-A}$.
    \item[\it Case 2] If $r \leq  L^{(k+1)(A+C)}$ and $|r\alpha -b| > y^{-1} x^{1-k} L^{(k+1)(A+C)}$, then $T_1(\alpha) \ll  yL^{-A}$.
    \item[\it Case 3] If $r \leq  L^{(k+1)(A+C)}$ and $|r\alpha -b| \leq  y^{-1} x^{1-k} L^{(k+1)(A+C)}$, we have
        \begin{eqnarray*}
          |ra-bq| &=& |r(a-q\alpha) + q(r\alpha-b)| \\
            &\leq& \frac{1}{Q} L^{(k+1)(A+C)} + Q y^{-1} x^{1-k} L^{(k+1)(A+C)} \\
            &\leq&  \frac{P L^{(k+1)(A+C)}}{x^{k-2}y^2} + \frac{y L^{(k+1)(A+C)}}{x P}.
        \end{eqnarray*}
        By (\ref{eqn: c_1. I}), we have $|ra-bq|<1$. Hence
        $$ a=b, \quad q=r.$$
        Then
        $$
          T_1(\alpha) \ll  \frac{w_k(q) q^\varepsilon L^C y}{1 + y x^{k-1} |\alpha - a/q|}.
        $$
  \end{description}
  So we have
  \begin{equation*}
    \mathcal{T}_1 \ll y L^{-A} + \frac{w_k(q) q^\varepsilon L^C y}{1 + y x^{k-1} |\alpha - a/q|}.
  \end{equation*}
  For $\alpha \in \mathcal{C}$, we have $q > P = L^{c_1}$.
  If we have $c_1$ as in (\ref{eqn: c_1. I}) then $\mathcal{T}_1 \ll y L^{-A}$.
\end{proof}

\begin{remark}
  One can estimate the following exponential sums of type I/II
  \begin{equation*}
    \sum_{m_1\sim M_1} \sum_{m_2\sim M_2} a(m_1,m_2) \sum_{x<m_1m_2n\leq x+y} e\left((m_1m_2n)^k \alpha\right)
  \end{equation*}
  with some suitable conditions on $M_1$ and $M_2$ as \cite[Lemma 3.2]{kumchev2013weyl} and \cite[Lemma 4.2]{wei2014sums} did, and give a better result than Lemma \ref{lemma: Type I}. 
\end{remark}

\subsection{Type II estimate}
To prove Theorem \ref{thm: exponential sums}, we also need to handle the exponential sums of type II.
Let $a(m)$ and $b(n)$ be arithmetic functions satisfying the property that for all natural numbers $m$ and $n$, one has
\begin{equation}
  a(m) \ll \tau^c(m) \quad \textrm{and} \quad  b(n) \ll \tau^c(n).
\end{equation}
Let $M$ and $N$ be positive parameters, and define the exponential sum
$\mathcal{T}_2 = \mathcal{T}_{2}(\alpha;M)$ by
\begin{equation}
  \mathcal{T}_2(\alpha;M) :=  \sum_{m\sim M} a(m) \sum_{ x< mn\leq x+y} b(n) e\left( (mn)^k \alpha \right).
\end{equation}
The following lemma gives an estimate for $\mathcal{T}_2$ which is an improvement of \cite[Lemma 3.1]{kumchev2013weyl}.

\begin{lemma} \label{lemma: Type II}
  Let $k,\gamma,\sigma_k$ be as in Lemma \ref{lemma: Type I}. Let $0 < \rho < \sigma_k /(8\gamma)$.
  Suppose that $\alpha \in \mathcal{C}$. 
  And let $x$ and $y$ be positive numbers with
  \begin{equation} \label{eqn: conditions of y. type II}
    y = x^\theta, \quad  \frac{1}{(1-2\rho)} \frac{3\gamma-\sigma_k-1}{2(2\gamma-\sigma_k-1)} \leq \theta \leq 1.
  \end{equation}
  Then
  \begin{equation*}
    \mathcal{T}_2 \ll y L^{-A}, 
  \end{equation*}
  provided that
  \begin{equation} \label{eqn: conditions of M. type II}
    x^{1/2} \leq M \ll x^{\theta-2\rho},
  \end{equation}
  and
  \begin{equation} \label{eqn: c_1. II}
    c_1 > 2(k+1)(A+C),
  \end{equation}
  where $C$ is a constant depending on $c$.
\end{lemma}

\begin{proof}
  Set $N = x/M$, $X=x/N$, and $Y = y/N = yM/x$. Define $\nu$ by $Y^\nu = x^{2\rho}L^{-1}$.
  By (\ref{eqn: conditions of M. type II}), we have
  $$\nu < \sigma_k/\gamma.$$
  For $n_1,\ n_2 \leq 2N$, let
  \begin{equation*}
    \mathcal{M}(n_1,n_2) = \{ m\in (M,2M] : x < mn_1, mn_2 \leq x+y \}.
  \end{equation*}
  By Cauchy's inequality and an interchange of the order of summation, we have
  \begin{equation} \label{eqn: mathcal{T}_2 to T_1}
    \mathcal{T}_2^2 \ll y^{1+\varepsilon} M + M L^C T_1(\alpha),
  \end{equation}
  where
  \begin{equation*}
    T_1(\alpha) = \sum_{n_1<n_2} \tau^c(n_1) \tau^c(n_2) \left| \sum_{m\in \mathcal{M}(n_1,n_2)} e\left( \alpha (n_2^k - n_1^k) m^k \right) \right|.
  \end{equation*}
  Let $\mathcal{N}$ denote the set of pairs $(n_1,n_2)$ with $n_1<n_2$ and $\mathcal{M}(n_1,n_2)\neq \varnothing$ for which there exist integers $b$ and $r$ such that
  \begin{equation} \label{eqn: conditions of r and b. Type II}
    1\leq r \leq Y^{k\nu}, \quad (b,r)=1, \quad |r(n_2^k-n_1^k)\alpha - b|\leq Y^{k\nu-1}X^{1-k}.
  \end{equation}
  Since $N/2< n_1 < n_2 \leq 2N$ and $\mathcal{M}(n_1,n_2)\neq \varnothing$, we have $n_2 - n_1 \leq yx^{-1}n_1$. Hence $\#\mathcal{N} \ll xyM^{-2}$. In order to handle the inner summation in $T_1(\alpha)$, we set
  \begin{equation*}
    \begin{split}
    X_1 &:= \max\left\{ M,\frac{x}{n_1} \right\} \asymp M = \frac{x}{N} = X, \\
    Y_1 &:= \min\left\{ 2M,\frac{x+y}{n_2} \right\} - \max\left\{ M,\frac{x}{n_1} \right\} \ll \frac{y}{N} = Y.
    \end{split}
  \end{equation*}
  If $Y_1 < X_1^{\gamma/(2\gamma-\sigma_k-1)}$, by (\ref{eqn: conditions of y. type II}) and (\ref{eqn: conditions of M. type II}), the contribution to $T_1(\alpha)$ is
  \begin{equation*}
    \ll  xy M^{-2} M^{\gamma/(2\gamma-\sigma_k-1)} \ll  y^{2-2\rho +\varepsilon} M^{-1}.
  \end{equation*}
  If $Y_1 \geq X_1^{\gamma/(2\gamma-\sigma_k-1)}$, since $\nu<\sigma_k/\gamma$, we can apply Lemma \ref{lemma: exponential sums in short intervals} with $\rho=\nu$, $x=X_1$, and $y=Y_1$ to the inner summation in $T_1(\alpha)$. We obtain
  \begin{equation} \label{eqn: T_1 to T_2}
    T_1(\alpha) \ll  y^{2-2\rho +\varepsilon} M^{-1} + T_2(\alpha),
  \end{equation}
  where
  \begin{equation*}
    \begin{split}
      T_2(\alpha) &= \sum_{(n_1,n_2)\in \mathcal{N}} \frac{\tau^c(n_1) \tau^c(n_2) w_k(r) Y_1}{1 + Y_1 X_1^{k-1} |(n_2^k - n_1^k)\alpha - b/r|} \\
      &\ll  \sum_{(n_1,n_2)\in \mathcal{N}} \frac{\tau^c(n_1) \tau^c(n_2) w_k(r) Y}{1 + Y X^{k-1} |(n_2^k - n_1^k)\alpha - b/r|}.
    \end{split}
  \end{equation*}

  We now change the summation variables in $T_2(\alpha)$ to
  \begin{equation*}
    d= (n_1,n_2), \quad n=n_1/d, \quad h = (n_2-n_1)/d.
  \end{equation*}
  We obtain
  \begin{equation} \label{eqn: T_2}
    T_2(\alpha) \ll \sum_{dh\leq y/M} {\sum_n}' \frac{\tau^c(nd) \tau^c(nd+hd) w_k(r) Y}{1 + Y X^{k-1} |h d^k R(n,h)\alpha - b/r|},
  \end{equation}
  where $R(n,h) = ((n+h)^k - n^k)/h$ and the inner summation is over $n$ with $(n,h)=1$ and $(nd,(n+h)d)\in \mathcal{N}$.
  For each pair $(d,h)$ appearing in the summation on the right-hand side of (\ref{eqn: T_2}), Dirichlet's lemma on rational approximations yields integers $b_1$ and $r_1$ with
  \begin{equation} \label{eqn: conditions of r_1 and b_1. Type II}
    1\leq r_1 \leq x^{-2k\rho} YX^{k-1}, \quad (b_1,r_1) = 1, \quad |r_1 hd^k\alpha - b_1| \leq x^{2k\rho} Y^{-1}X^{1-k}.
  \end{equation}
  As $R(n,h) \leq 4^k (N/d)^{k-1}$, combining (\ref{eqn: conditions of M. type II}), (\ref{eqn: conditions of r and b. Type II}) and (\ref{eqn: conditions of r_1 and b_1. Type II}), we have
  \begin{equation*}
    \begin{split}
      | b_1 r R(n,h) - br_1 | & = |rR(n,h)(b_1-r_1hd^k\alpha) + r_1(rhd^kR(n,h)\alpha-b)| \\
        &\leq  r_1 Y^{k\nu-1}X^{1-k} + r R(n,h) x^{2k\rho} Y^{-1}X^{1-k} \\
        &\leq  L^{-k} + 4^k N^{k-1} x^{2k\rho} L^{-k} x^{2k\rho} Y^{-1}X^{1-k} <1.
    \end{split}
  \end{equation*}
  Hence,
  \begin{equation} \label{eqn: b/r and r}
    \frac{b}{r} = \frac{b_1 R(n,h)}{r_1}, \quad r = \frac{r_1}{(r_1,R(n,h))}.
  \end{equation}
  Combining (\ref{eqn: T_2}) and (\ref{eqn: b/r and r}), we obtain
  \begin{equation*}
    T_2(\alpha) \ll \sum_{dh\leq y/M}  \frac{\tau^{2c}(d) Y}{1 + Y X^{k-1} N_d^{k-1}|h d^k \alpha - b_1/r_1|} \sum_{\substack{n\sim N_d \\ (n,h)=1}} \tau^{c}(n) \tau^c(n+h) w_k\left( \frac{r_1}{(r_1,R(n,h))} \right),
  \end{equation*}
  where $N_d = N/d$. By Lemma \ref{lemma: some inequalities}, we deduce that
  \begin{equation} \label{eqn: T_2 to T_3}
    T_2(\alpha) \ll y^2 x^{-1+\varepsilon} + T_3(\alpha),
  \end{equation}
  where
  \begin{equation*}
    \begin{split}
      T_3(\alpha) & = \sum_{dh\leq y/M} \frac{w_k(r_1) r_1^\varepsilon L^c \tau^{c}(d) Y N_d}{1 + Y X^{k-1} N_d^{k-1}|h d^k \alpha - b_1/r_1|} \\
        & \ll \sum_{dh\leq y/M} \frac{r_1^\varepsilon L^c \tau^{c}(d) Y N_d}{(r_1 + Y X^{k-1} N_d^{k-1}|r_1 h d^k \alpha - b_1|)^{1/k}}.
    \end{split}
  \end{equation*}

  We now write $\mathcal{H}$ for the set of pairs $(d,h)$ with $dh \leq y/M$ for which there exist integers $b_1$ and $r_1$ subject to
  \begin{equation} \label{eqn: restriction conditions of r_1 and b_1. Type II}
    1\leq r_1 \leq x^{2k\rho}, \quad  (b_1,r_1) = 1, \quad  |r_1hd^k\alpha - b_1| \leq x^{-k+1+2k\rho} Y^{-1}.
  \end{equation}
  We have
  \begin{equation} \label{eqn: T_3 to T_4}
    T_3(\alpha) \ll  y^{2-2\rho+\varepsilon} M^{-1} + T_4(\alpha),
  \end{equation}
  where
  \begin{equation*}
    T_4(\alpha) = \sum_{(d,h)\in \mathcal{H}} \frac{w_k(r_1) r_1^\varepsilon L^c \tau^{c}(d) Y N_d}{1 + Y X^{k-1} N_d^{k-1}|h d^k \alpha - b_1/r_1|}.
  \end{equation*}
  For each $d\leq y/M$, Dirichlet's lemma on rational approximations yields integers $b_2$ and $r_2$ with
  \begin{equation} \label{eqn: conditions of r_2 and b_2. Type II}
    1\leq r_2 \leq x^{k-1-2k\rho} Y/2, \quad  (b_2,r_2) = 1, \quad  |r_2 d^k \alpha - b_2|\leq 2 x^{-k+1+2k\rho} Y^{-1}.
  \end{equation}
  Combining (\ref{eqn: restriction conditions of r_1 and b_1. Type II}) and (\ref{eqn: conditions of r_2 and b_2. Type II}), we obtain
  \begin{equation*}
    \begin{split}
      |b_2r_1h - b_1r_2| & = |r_1h(b_2-r_2d^k\alpha)+r_2(r_1hd^k\alpha-b_1)| \\
          &\leq r_1 h |r_2d^k\alpha - b_2| + r_2 |r_1hd^k\alpha - b_1| \\
          &\leq 1/2 + 2 x^{-k+2+4k\rho} M^{-2} < 1,
    \end{split}
  \end{equation*}
  whence
  \begin{equation*}
    \frac{b_1}{r_1} = \frac{hb_2}{r_2}, \quad r_1 = \frac{r_2}{(r_2,h)}.
  \end{equation*}
  We write $Z_d = Y X^{k-1} N_d^{k-1} |d^k\alpha - b_2/r_2|$ and by Lemma \ref{lemma: some inequalities}, we obtain
  \begin{equation*}
    \begin{split}
      T_4(\alpha) & = \sum_{(d,h)\in \mathcal{H}}  \frac{r_2^\varepsilon L^c \tau^{c}(d) Y N_d}{1 + Z_d h} w_k\left( \frac{r_2}{(r_2,h)} \right) \\
       & \ll \sum_{d\leq y/M} r_2^\varepsilon L^c \tau^c(d) yd^{-1} L \max_{1\leq H \leq \frac{y}{Md}} \sum_{h\sim H} \frac{1}{1 + Z_d h} w_k\left( \frac{r_2}{(r_2,h)} \right) \\ 
       & \ll \sum_{d\leq y/M} \frac{r_2^\varepsilon L^c \tau^c(d) w_k(r_2) y^2 M^{-1}}{d^2 (1 + y(Md)^{-1} Z_d )}.
    \end{split}
  \end{equation*}
  Hence
  \begin{equation} \label{eqn: T_4 to T_5}
    T_4(\alpha) \ll y^{2-2\rho+\varepsilon}M^{-1} + T_5(\alpha),
  \end{equation}
  where
  \begin{equation*}
    T_5(\alpha) = \sum_{d\in \mathcal{D}} \frac{r_2^\varepsilon L^c \tau^c(d) w_k(r_2) y^2 M^{-1}}{d^2 (1 + y^2 x^{k-2} d^{-k} |d^k\alpha - b_2/r_2|)},
  \end{equation*}
  and $\mathcal{D}$ is the set of integers $d\leq x^{2\rho}$ for which there exist integers $b_2$ and $r_2$ with
  \begin{equation} \label{eqn: restriction conditions of r_2 and b_2. Type II}
    1\leq r_2 \leq x^{2k\rho}, \quad  (b_2,r_2)=1, \quad  |r_2 d^k \alpha - b_2| \leq y^{-2} x^{2-k} L^{(k+1)(2A+C)}.
  \end{equation}
  Combining (\ref{eqn: PQR}), (\ref{eqn: Dirichlet's theorem on rational approximations}) and (\ref{eqn: restriction conditions of r_2 and b_2. Type II}), we deduce that
  \begin{equation*}
    \begin{split}
      |r_2 d^k a - b_2 q| & = |r_2 d^k (a-q\alpha) + q(r_2 d^k \alpha - b_2)| \\
          & \leq  r_2 d^k Q^{-1} + q |r_2 d^k \alpha - b_2| \\
          & \leq  x^{4k\rho} Q^{-1} + y^{-2} x^{2-k} L^{(k+1)(2A+C)} Q < 1,
    \end{split}
  \end{equation*}
  whence
  \begin{equation*}
    \frac{b_2}{r_2} = \frac{d^k a}{q} , \quad  r_2 = \frac{q}{(q,d^k)}.
  \end{equation*}
  Thus, recalling Lemma \ref{lemma: some inequalities}, we obtain
  \begin{equation} \label{eqn: T_5}
    \begin{split}
      T_5(\alpha) &\ll  \frac{q^\varepsilon L^c y^2 M^{-1}}{1+ y^2 x^{k-2} |\alpha - a/q|}  \sum_{d\leq x^{2\rho}} \tau^c(d) d^{-2} w_k\left( \frac{q}{(q,d^k)} \right)  \\
      &\ll  \frac{q^\varepsilon L^C w_k(q) y^2 M^{-1}}{1+ y^2 x^{k-2} |\alpha - a/q|}.
    \end{split}
  \end{equation}

  The desired estimate follows from (\ref{eqn: conditions of y. type II}), (\ref{eqn: conditions of M. type II}), (\ref{eqn: mathcal{T}_2 to T_1}), (\ref{eqn: T_1 to T_2}), (\ref{eqn: T_2 to T_3}), (\ref{eqn: T_3 to T_4}), (\ref{eqn: T_4 to T_5}), and (\ref{eqn: T_5}).
\end{proof}

\subsection{Complete the proof of Proposition \ref{prop: C}}
We now deduce Proposition \ref{prop: C} from Lemmas \ref{lemma: Type I}, \ref{lemma: Type II} and Vaughan's identity for $\mu(n)$.

  We put
  \begin{equation} \label{eqn: U & V}
    U = x^{\theta/2 - \rho}, \quad  V = x^{1 - \theta + 2\rho}.
  \end{equation}
  Take
  \begin{equation} \label{eqn: the value of rho}
    \rho = \frac{1}{2} \min\left\{ \frac{\sigma_k}{8\gamma}, \frac{1}{2}\left(\theta-\frac{2}{3}\right) \right\}.
  \end{equation}
  We have
  \begin{equation} \label{eqn: UV & x/U}
    UV \asymp (x+y)/U \asymp x^{1-\theta/2 + \rho} \ll x^{\theta-2\rho}.
  \end{equation}
  And then we apply Vaughan's identity as in Lemma \ref{lemma: Vaughan's identity}.
  Thus we deduce that
  \begin{equation} \label{eqn: S_k to -S_1+S_2}
    S_k(x,y;\alpha) = - S_1 + S_2,
  \end{equation}
  where
  \begin{eqnarray*}
    S_1 &=& \sum_{1\leq v\leq UV} \lambda_0(v) \sum_{x< lv \leq x+y} e\left((lv)^k\alpha\right),\\
    S_2 &=& \sum_{V < u\leq (x+y)/U} \lambda_1(u) \sum_{\substack{x< mu \leq x+y \\ m> U}} \mu(m) e\left((mu)^k\alpha\right),
  \end{eqnarray*}
  in which
  \begin{equation*}
    \lambda_0(v) = \sum_{\substack{md=v \\ 1\leq d\leq V \\ 1\leq m\leq U}} \mu(d)\mu(m)
    \quad \textrm{and} \quad
    \lambda_1(u) = \sum_{\substack{d|u \\ d > V}} \mu(d).
  \end{equation*}

  We begin with estimating the sum $S_2$. Take
  \begin{equation} \label{eqn: d}
    \gamma = (\theta - 3/4)^{-1}.
  \end{equation}
  Since $3/4<\theta\leq 1$, by (\ref{eqn: the value of rho}) we have
  \begin{equation*}
    \frac{1}{(1-2\rho)} \frac{3\gamma-\sigma_k-1}{2(2\gamma-\sigma_k-1)} \leq \theta \leq 1.
  \end{equation*}
  To apply Lemma \ref{lemma: Type II}, we further divide $S_2$ into to two parts
  \begin{equation*}
    S_{21} = \sum_{x^{1/2} \leq u \leq (x+y)/U} \lambda_1(u) \sum_{\substack{x< mu \leq x+y \\ m> U}} \mu(m) e\left((mu)^k\alpha\right),
  \end{equation*}
  and
  \begin{equation*}
    S_{22} = \sum_{V < u < x^{1/2}} \lambda_1(u) \sum_{\substack{x< mu \leq x+y \\ m> U}} \mu(m) e\left((mu)^k\alpha\right).
  \end{equation*}
  On noting that (\ref{eqn: U & V}), (\ref{eqn: UV & x/U}) and $\lambda_1(u) \leq \tau(u)$,
  we can divide the summation over $u$ into dyadic intervals to deduce from Lemma \ref{lemma: Type II} that
  \begin{equation*}
    S_{21} \ll  (\log x) \max_{x^{1/2}\leq M\leq (x+y)/U} \left| \sum_{u \sim M} a(u) \sum_{x< mu \leq x+y} b(m) e\left((mu)^k\alpha\right) \right|
           \ll  y L^{-A},
  \end{equation*}
  where $a(u) = \lambda_1(u)$, and $b(m) = \mu(m)$ if $m>U$ and is 0 if else.
  For $S_{22}$, we first interchange the order of summation, and then by the same argument as above, we obtain
  \begin{equation*}
    S_{22} \ll y L^{-A}.
  \end{equation*}
  Hence we obtain
  \begin{equation} \label{eqn: S_2}
    S_{2} \ll y L^{-A}.
  \end{equation}

  Next we estimate $S_1$. Write
  \begin{equation*}
    S_3(Z,W) = \sum_{Z < v \leq W} \lambda_0(v) \sum_{x< lv \leq x+y} e\left((lv)^k\alpha\right).
  \end{equation*}
  Then we find that
  \begin{equation} \label{eqn: S_1 to S_3}
    S_1 = S_3(0,V) + S_3(V,UV).
  \end{equation}
  Note that (\ref{eqn: U & V}), (\ref{eqn: UV & x/U}) and the bound $|\lambda_0(v)| \leq \tau(v)$, we deduce from Lemma \ref{lemma: Type II} that
  \begin{equation} \label{eqn: S_3(V,UV)}
    S_3(V,UV) \ll y L^{-A}.
  \end{equation}
  We then estimate $S_3(0,V)$.
  Since $3/4<\theta\leq 1$, by (\ref{eqn: U & V}), (\ref{eqn: the value of rho}) and (\ref{eqn: d}), we have
  \begin{equation*}
    V \ll y \left( \frac{y}{x} \right)^{\frac{\gamma+1}{\gamma-\sigma_k-1}}, \quad
    V \ll y x^{-\gamma\rho/\sigma_k}, \quad
    V^{2k} \ll y x^{k-1-2k\rho}.
  \end{equation*}
  So we can divide the summation over $v$ into dyadic intervals to deduce from Lemma \ref{lemma: Type I} that
  \begin{equation} \label{eqn: S_3(0,V)}
    S_3(0,V) \ll y L^{-A}.
  \end{equation}
  Thus, by combining (\ref{eqn: S_3(V,UV)}) and (\ref{eqn: S_3(0,V)}), we deduce form (\ref{eqn: S_1 to S_3}) that
  \begin{equation} \label{eqn: S_1}
    S_1 \ll y L^{-A}.
  \end{equation}

  Proposition \ref{prop: C} follows from (\ref{eqn: S_k to -S_1+S_2}), (\ref{eqn: S_2}) and (\ref{eqn: S_1}).


\medskip
\noindent
{\sc Acknowledgements.} The author would like to thank Professor Jianya Liu
for his valuable advice and constant encouragement. He also want to thank
the referees and editors for their kind comments and valuable suggestions.

\medskip
\bibliographystyle{plain}
\bibliography{hbrbib_exponentialsums}

\end{document}